\documentclass[12pt]{amsart} 
\usepackage{euler}
\usepackage{euscript}
\usepackage{multicol} 
\usepackage{amssymb}

\usepackage{color}

\definecolor{amaranth}{rgb}{0.9, 0.17, 0.31}
\definecolor{ao}{rgb}{0.0, 0.0, 1.0}
\definecolor{ao(english)}{rgb}{0.0, 0.5, 0.0}
\definecolor{deepmagenta}{rgb}{0.8, 0.0, 0.8}

\usepackage{enumitem}

\usepackage[pdftex]{graphicx}
\usepackage{amsmath} 
\usepackage{times} 
 
\def\XXint#1#2#3{{\setbox0=\hbox{$#1{#2#3}{\int}$}
     \vcenter{\hbox{$#2#3$}}\kern-.5\wd0}}

\theoremstyle{plain}
\newtheorem{theorem}{Theorem}[section]
\newtheorem{proposition}{Proposition}[section]

\newtheorem{remark}{\bf Remark}[section]
\theoremstyle{definition}

\newcommand{\C}{\mathbb{C}}

\newcommand{\dd}{\, \mathrm{d}}
\newcommand{\eps}{\varepsilon}
\renewcommand{\epsilon}{\varepsilon}
\newcommand{\e}{\mathrm{e}}
\newcommand{\ii}{\mathrm{i}}

\renewcommand{\phi}{\varphi}

\newcommand{\set}[2]{\left\{#1: \, #2\right\}}

\DeclareMathOperator{\dom}{dom}
\DeclareMathOperator{\im}{Im}

\DeclareMathOperator{\re}{Re}

\dedicatory{To Dima Yafaev}

\title
[Eigenvalues of non-selfadjoint functional difference operators]
 {Eigenvalues of non-selfadjoint functional difference operators}

\author{Alexei Ilyin}
\address{Alexei Ilyin:  Keldysh Institute of Applied Mathematics;
ilyin@keldysh.ru}

\author{Ari Laptev}
\address{Ari Laptev:  Department of Mathematics, Imperial College London,
 and  Sirius University of Science and Technology  Olimpiyskiy ave. b.1,
 Sirius, Krasnodar region, Russia, 354340;
a.laptev@imperial.ac.uk}

\author{Lukas Schimmer}
\address{Lukas Schimmer:  Department of Mathematical Sciences, Loughborough University, Loughborough,
Leicestershire, LE11 3TU, United Kingdom
l.schimmer@lboro.ac.uk}

\author{Anna Zernova}
\address{Anna Zernova: Sirius University of Science and Technology  Olimpiyskiy ave. b.1,
 Sirius, Krasnodar region, Russia, 354340;
 yakimenkoanyuta@gmail.com }

\begin{document}

\maketitle

\begin{quote}
{\normalfont\fontsize{8}{10}
\selectfont{\bfseries  Abstract.}
Using the well known approach developed in the papers of B.Davies and his co-authors we obtain inequalities 
for the location of possible complex eigenvalues of non-selfadjoint functional difference operators. When studying the sharpness of the main result we discovered that complex potentials can create resonances.
}
\end{quote}

\setcounter{equation}{0}
\section{Introduction}\label{Sec:1}

\noindent
In this paper we are concerned with possible locations of eigenvalues of non-selfadjoint functional difference operators with complex-valued potentials.

\noindent
Let $P$ be the self-adjoint quantum mechanical momentum operator on $L^2(\Bbb R)$, i.e.~$P=\ii \frac{d}{dx}$ and for $b>0$ denote by $U(b)$ the Weyl operator $U(b)=\exp(-bP)$. 
By using the Fourier transform
\begin{align*}
\widehat{\psi}(k)=(\mathcal{F}\psi)(k)=\int_{\Bbb R} \e^{-2\pi \ii kx}\psi(x)\dd x
\end{align*}
we can describe the domain of $U(b)$ as 
\begin{align*}
\dom(U(b))=\set{\psi\in L^2(\Bbb R)}{\e^{-2\pi b k}\widehat{\psi}(k)\in L^2(\Bbb R)}.
\end{align*}
This set consists of those functions $\psi(x)$ that admit an analytic continuation to the strip $\set{z = x+\ii y\in\Bbb C}{0<y <b}$ such that $\psi(x + \ii y) \in L^2(\Bbb R)$ for all $0\leq y < b$ and there is a limit $\psi(x + \ii b -  \ii 0) = \lim_{\eps\to 0^+}\psi (x + \ii b - \ii \eps)$ in the sense of convergence in $L^2(\Bbb R)$, which we will denote simply by 
$\psi(x+ \ii b)$. The domain of the inverse operator $U^{-1}(b)$ can be characterised similarly. 

\noindent
For $b>0$ we define the operator $W_0(b)=U(b)+U(b)^{-1}=2\cosh(bP)$ on the domain
\begin{align*}
\dom(W_0(b))=\set{\psi\in L^2(\Bbb R)}{2\cosh(2\pi bk)\widehat{\psi}(k)\in L^2(\Bbb R)}.
\end{align*}
The operator $W_0(b)$ is self-adjoint and unitarily equivalent to the multiplication operator $2\cosh(2\pi b k)$ in the Fourier space. Its spectrum is thus absolutely continuous covering the interval
$[2,\infty)$ doubly.

\noindent
In this paper our aim is to obtain an estimate for complex eigenvalues of the operator 
\begin{equation}\label{W_V}
W_V(b) = W_0(b) - V,
\end{equation}
where the potential $V$ is a complex-valued function. 

\medskip
\noindent
In order to describe our result, we first assume that $V\in L^1(\Bbb R)$ is real-valued.  The scalar inequality 
$2\cosh(2\pi b k)-2\ge (2\pi b k)^2$ implies the operator inequality
\begin{align}\label{WH}
W_0(b)-2\ge-b^2\frac{d^2}{dx^2}
\end{align}
on $\dom(W_0(b))$. By Sobolev's inequality, we can conclude that the operator  \eqref{W_V}
is bounded from below on the common domain of  $W_0(b)$ and $V$. 
We can thus consider its Friedrichs extension, which we continue to denote by $W_V(b)$.  
By applying Weyl's theorem (in a version for quadratic forms) and Rellich's lemma together with the fact that the form domain of $W_0(b)$ is continuously embedded in $H^1(\Bbb R)$ we conclude that the spectrum of $W_V(b)$ consists of essential spectrum $[2,\infty)$ and discrete finite-multiplicity eigenvalues  below. Details of this argument in the similar case of a Schr{\"o}dinger operator can be found in the book \cite{FLW}.

\medskip
\noindent
Any eigenvalue $\lambda$ of the operator \eqref{W_V} with real-valued $V$ can be written as $\lambda = -2\cos(\omega)$,  with $\omega\in [0,\pi)$ for $\lambda\in[-2,2]$ and  $\omega\in \ii\,[0,\infty)$ for $\lambda\le -2$. Under the condition that all eigenvalues $\lambda_j=-2\cos(\omega_j)$ are larger than or equal to $-2$, the authors of \cite{LSch} proved a Lieb--Thirring inequality 
\begin{equation*}\label{main_real}
\sum_{j\ge 1}\frac{\sin(\omega_j)}{\omega_j} \le \frac{1}{2\pi b} \int_{\Bbb R} |V(x)|\dd x.
\end{equation*}
As discussed in  \cite[Remark 1.2]{LSch}, the proof in general does not apply if there are multiple eigenvalues below $-2$. However, in the special case that single one of the eigenvalues is below $-2$ the proof remains applicable. Furthermore, it can also be used to establish that any real eigenvalue $\lambda=-2\cos(\omega)$, regardless of whether it lies above or below $-2$, must satisfy
\begin{equation}\label{main_real}
\frac{\sin(\omega)}{\omega} \le \frac{1}{2\pi b} \int_{\Bbb R} |V(x)|\dd x.
\end{equation}
The constant $\frac{1}{2\pi b}$in this inequality is sharp and attained if $V (x) = c\delta(x)$, $c > 0$. 

\medskip
\noindent
In recent years there has been an increasing interest in eigenvalue estimates for complex-valued potentials. The authors in \cite{AAD} developed an elegant observation that allows to locate complex eigenvalues for Schr\"odinger operators with complex-valued potentials. Such an approach and its generalisations were used in \cite{FLS}, \cite{CLT}, \cite{Fr1}.
Further development of estimates of complex eigenvalues for Schr\"odinger operators were obtained in \cite{Fr2}, \cite{LSaf}, \cite{BC}, \cite{FS}, \cite{Saf} and many others.

\medskip
\noindent 
It turns out that the inequality \eqref{main_real} can be generalised to the non-selfadjoint case. 

\noindent 
Let  $\Bbb R_+ = [0,\infty)$ and $\Bbb R_-= (-\infty,0]$. Denote by 
\begin{equation}\label{Omeg}
\Omega = \{\omega\in \Bbb C:\, {\rm Re}\, \omega\in[0,\pi);\, \, {\rm Im}\, \omega \in \Bbb R\}
\end{equation}
and 
\begin{equation}\label{Omeg_pm}
\Omega_\pm = \{\omega\in \Bbb C:\, {\rm Re}\, \omega\in[0,\pi);\, \, {\rm Im}\, \omega \in \Bbb R_\pm\}
\end{equation}
Then the mapping $\omega \mapsto \lambda(\omega) = -2\cos(\omega)$ transfers $\Omega$ to $\Bbb C \setminus [2,\infty)$ and $\Omega_\pm$ to $ \Bbb C_\pm\setminus [2,\infty)$, where $\Bbb C_\pm=\{ z\in\Bbb C:\,{\rm Im} \, z \in \Bbb R_\pm\}$.

\medskip
\noindent
Our main result is the following. 

\begin{theorem}\label{1}
Let $V\in L^1(\Bbb R)$ be a complex-valued potential. Then the eigenvalues $\lambda\in \Bbb C\setminus [2,\infty)$ of the operator $W_V(b)$ satisfy the inequality 
\begin{equation}\label{main}
\left| \frac{\sin(\omega)}{\omega}\right| \le \frac{1}{2\pi b} \int_{\Bbb R} |V(x)|\dd x,
\end{equation}
where $\lambda = - 2\cos(\omega)$ and where $ \omega \in \Omega$.

\noindent
The constant in this inequality is sharp in the sense that there are potentials
$V$ such that inequality \eqref{main} becomes an equality.
\end{theorem}

\medskip
\noindent
The study of different aspects of the spectrum of functional difference operators $W_V(b)$ was considered before. 
In the case when $-V =V_0= \e^{2\pi b x}$ is an exponential function, the operator $W_{V}(b)$ first appeared in the study of the quantum Liouville model on the lattice \cite{FT1} and plays an important role in the representation theory of the non-compact quantum group $\mathrm{SL}_{q}(2,\Bbb R)$. The spectral analysis of this operator was studied in \cite{FT2}. 
In the case when $-V = 2\cosh (2\pi b x)$ the spectrum of $W_{V}(b)$ is discrete and converges to $+\infty$. Its Weyl asymptotics were obtained in \cite{LSchTI}. This result was extended to a class of growing potentials in \cite{LSchTII}. More information on spectral properties of functional difference operators can be found in papers \cite{GHM}, \cite{GKMR}, \cite{KM}, \cite{KMZ}, \cite{T}.


\bigskip
\setcounter{equation}{0}
\section{Resonance state}\label{Sec:2}

\noindent
We begin by proving that in the self-adjoint case the spectral point $2$ is the resonance state for the operator 
\eqref{W_V}. 

\begin{theorem}\label{VSt}
Let $W_V$ defined in \eqref{W_V} be a self-adjoint, semi-bounded operator such that $V\ge0$, $V\not\equiv 0$, $V\in L^1(\Bbb R)$. Then $W_V$ has at least one eigenvalue below the spectral point $2$. 
\end{theorem}

\begin{remark}
It is well known that for a one-dimensional Schr\"odinger operator $-d^2/dx^2 -V$, $V\ge0$, $V\not\equiv 0$, there is always at least one negative eigenvalue. Since we have the strict inequality $W_0 - 2 > -d^2/dx^2$, Theorem \ref{VSt} cannot be obtained directly from the mentioned result for Schr\"odinger operators.
\end{remark}

\begin{proof}
For the proof we consider the sequence of test functions 
$$
u_n(x) = \e^{- \frac{x^2}{n^2}} \in {\rm dom}(W_V), \quad x\in\Bbb R.
$$
Clearly for any fixed $x\in \Bbb R$ we have $u_n \to 1$ as $n\to\infty$. Applying the Fourier transform we obtain
$$
\widehat{u}_n(k) = (\mathcal F u_n) (k) = \int_{\Bbb R} \e^{-2\pi \ii k x} \e^{- \frac{x^2}{n^2}} \dd x 
=\sqrt\pi \,n\, \e^{-\pi^2 n^2 k^2}
$$
and hence
\begin{align*}
((W_V - 2)  u_n, u_n) &= \int_{\Bbb R}  \left((W_0-2) u_n\right) \, \overline{u_n} \, dx - \int_{\Bbb R} V |u_n|^2 \dd x \\
&= 
\sqrt\pi \,n\, \int_{\Bbb R} \left(2\cosh(2\pi b k)-2\right) \e^{-2\pi^2 n^2 k^2}\, dk - \int_{\Bbb R} V |u_n|^2 \dd x .
\end{align*}
\noindent
Since 
$$
n \, \int_{\Bbb R} \left(2\cosh(2\pi b k)-2\right) \e^{-2\pi^2 n^2 k^2}\dd k \to 0, \quad {\rm as} \quad n\to \infty,
$$
we have that there is $n_0$ such that for any $n>n_0$ 
$$
((W_V - 2)  u_n, u_n) < 0.
$$
Applying the variational principle we complete the proof.
\end{proof}


\medskip
\setcounter{equation}{0}
\section{Free resolvent}\label{Sec:3}

\noindent
Since the spectrum $\sigma(W_0(b)) =[2,\infty)$ we conclude that $W_0(b)-\lambda$ is an invertible operator for $\lambda\in\Bbb C\setminus [2,\infty)$.  
Let  as before
$\lambda=-2\cos(\omega)$ with $\omega\in\Omega$. Then in Fourier space the inverse of $W_0(b)-\lambda$ is given by the multiplication operator $(2\cosh(2\pi b k)+2\cos(\omega))^{-1}$. 

\medskip
\noindent
Applying the inverse Fourier transform $\mathcal{F}^{-1}$ to $(2\cosh(2\pi b k)+2\cos(\omega))^{-1}$ we find the kernel of the free resolvent $G_\lambda=(W_0(b)-\lambda)^{-1}$ that is 
\begin{equation}\label{rezolvent}
G_\lambda(x,y)  = G_\lambda(x-y)=\frac{1}{2b\sin(\omega)}\, \frac{\sinh \left(\frac{\omega}{b} (x-y)\right)}{\sinh \left(\frac{\pi}{b} (x-y)\right)}.
\end{equation}
In the derivation of this identity using Contour integration, it is essential that $0\le\re\omega<\pi$. If $\omega$ had for example been chosen such that $\pi\le\re\omega<2\pi$,  the factor $\omega$ in \eqref{rezolvent} would have to be replaced by $\omega-2\pi$, guaranteeing again an exponential decay.  
\begin{remark}
Note that $G_\lambda(x-y)$ is an even and positive kernel for $\omega\in[0,\pi)$ and it becomes oscillating if 
$\omega\in i (-\infty,\infty)$. 
\end{remark}

\medskip
\noindent
The value of $G_\lambda$ on the diagonal $x=y$ takes the form
\begin{align}\label{G0}
G_\lambda(0)=\frac{1}{2\pi b}\frac{\omega}{\sin(\omega)}\,
\end{align}
and we can see the relation between the right-hand side of \eqref{G0} and  the expression in the left-hand sides of inequalities \eqref{main_real} and \eqref{main}.
Due to our parameterisation of the spectral parameter,  the convergence $\lambda\to 2$ in $\Bbb C\setminus[0,\infty)$ implies $\omega\to \pi$ in $\Omega$ and thus 
\begin{align*}
G_\lambda(0) \sim \frac{1}{2b} \, \frac{1}{\sqrt{1-\cos^2\omega}} \sim  \frac{1}{2b} \,\frac{1}{\sqrt{2-\lambda}},
\quad {\rm as} \quad \lambda\to 2.
\end{align*}
If $|\lambda|\to\infty$, then $|\im\omega| \to  \infty$ and 
\begin{align*}
|G_\lambda(0)| \sim \frac{1}{\pi b} |\lambda|^{-1} \log |\lambda|.
\end{align*}
%

\begin{proposition} \label{G_ineq}
For any $\lambda \in \Bbb C\setminus [2,\infty)$ we have
\begin{equation}\label{G_inequality}
|G_\lambda (x)|\le |G_\lambda(0)|, \quad \forall x\in \Bbb R.
\end{equation}
\end{proposition} 

\begin{proof}
In order to prove \eqref{G_inequality} it is enough to show 
$$
\left|\frac{\sinh \left(\frac{\omega}{b} x\right)}{\sinh \left(\frac{\pi}{b} x\right)}\right|
\le \frac{|\omega|}{\pi},
$$
where $\omega\in\Omega$ as defined in \eqref{Omeg}.
We first prove that for any $\alpha\in \C$ with $0\le\re \alpha\le 1$ and any $x\in\Bbb R$
\begin{align}
|\cosh(\alpha x)|\le \cosh(x)\,.
\label{eq:cosh}
\end{align}
It suffices to consider  $x\ge0$. We define the holomorphic function $g(\alpha)=\cosh(\alpha x)/\cosh(x)$ on the strip $0<\re \alpha< 1$. Clearly it has a continuous extension to $\re\alpha=0$ and $\re\alpha=1$. On these boundaries it holds that $|g(\alpha)|\le 1$ since for any $t\in\Bbb R$ 
\begin{align*}
|g(0+\ii t)|=\frac{|\cosh(i t x)|}{\cosh(x)}=\frac{|\cos(t x)|}{\cosh(x)}\le 1
\end{align*}
and
\begin{align*}
|g(1+\ii t)|^2
&=\frac{|\cosh(x)\cos(tx)+\ii\sinh(x)\sin(tx)|^2}{\cosh^2(x)}\\
&=\cos^2(tx)+\tanh^2(x)\sin^2(tx)\le 1\,.
\end{align*}
On the interior $0<\re \alpha< 1$ the function is furthermore bounded
\begin{align*}
|g(\alpha)|=\frac{|\e^{\alpha x}+\e^{-\alpha x}|}{\e^{x}+\e^{-x}}
=\e^{(\re \alpha-1)x}\frac{|1+\e^{-2\alpha x}|}{1+\e^{-2 x}}
\le 1+\e^{-2\re\alpha x}\le 2\,.
\end{align*}
By the Hadamard three-lines theorem (or the Phragm{\'e}n--Linedl{\"o}f principle on vertical strips), we have that $|g(\alpha)|\le 1$ for all $\alpha$ with $0\le\re \alpha\le 1$, which proves \eqref{eq:cosh}.

As a consequence for any such $\alpha\neq0$ and any $y\in\Bbb R\setminus\{0\}$
\begin{align*}
\left|\frac{\sinh(\alpha y)}{\alpha y}\right|
=\left|\int_0^1\cosh(\alpha y t)\dd t\right|
&\le \int_0^1|\cosh(\alpha y t)|\dd t\\
&\le \int_0^1\cosh(y t)\dd t=\frac{\sinh(y)}{y}=\left|\frac{\sinh(y)}{y}\right|\,.
\end{align*}
Applying this result with $\alpha=\omega/\pi$  and $y=\pi x/b$ we obtain that 
\begin{align*}
\left|\frac{\sinh(\frac{\omega}{b} x)}{\omega x}\right|
\le \left| \frac{\sinh(\frac{\pi}{b}x)}{\pi x}\right|
\end{align*}
for all $\omega\neq0$ with $0\le\re\omega\le \pi$ and all $x\in \Bbb R\setminus\{0\}$. Rearranging yields the desired
result and the proof is complete.
\end{proof}

\medskip
\noindent
Note that in \cite{FT2} L.~Faddeev and L.~A.~Takhtajan studied the resolvent in a slightly different form 
\begin{align*}
G_\lambda(x-y) = \frac{\sigma}{\sinh (\frac{\pi \ii \varkappa}{\sigma}) }\, \left( \frac{\e^{-2\pi \ii \varkappa(x-y)}}{1-\e^{- 4 \pi \ii \sigma (x-y)}} + 
\frac{\e^{2\pi \ii \varkappa(x-y)}}{1-\e^{4 \pi \ii \sigma (x-y)}} \right)
\end{align*}
which coincides with \eqref{rezolvent} with $\sigma = \ii/2b$, $\lambda = 2\cosh(2b\pi\varkappa)$  and $\varkappa = \frac{\omega - \pi}{2\pi \ii b}$.

\medskip
\noindent
It was also pointed out in \cite{FT2} that the free resolvent can be written using the analogues of the Jost solutions 
\begin{align*}
f_-(x,\varkappa) = \e^{-2\pi \ii \varkappa x} \quad {\rm and} \quad f_+(x,\varkappa) = \e^{2\pi \ii \varkappa x}
\end{align*}
that appear in the theory of one-dimensional Schr{\"o}dinger operators. Namely
\begin{align*}
G_\lambda(x-y) = \frac{2\sigma}{C(f_-,f_+)(\varkappa)} \, 
\left(\frac{f_-(x, \varkappa) f_+(y, \varkappa)}{1 - \e^{\frac{\pi \ii}{\sigma'} (x-y)}}  + 
\frac{f_-(y, \varkappa) f_+(x, \varkappa)}{1 - \e^{-\frac{\pi \ii}{\sigma'} (x-y)}} \right),
\end{align*}
where $\sigma'  \sigma = -1/4$ and where $C(f,g)$ is the so-called Casorati determinant (a difference analogue of the Wronskian) of the solutions of the functional-difference equation
\begin{align*}
C(f,g)(x,\varkappa) = f(x+ 2\sigma', \varkappa)g(x,\varkappa) - f(x,\varkappa) g(x + 2\sigma', \varkappa).
\end{align*}
For the Jost solutions we have $C(f_-,f_+)(x,\varkappa) = 2\sinh(\frac{2\pi \kappa}{\sigma})$.


\medskip
\setcounter{equation}{0}
\section{Proof of Theorem \ref{1}}\label{Sec:4}

\noindent
Let $V\in L^1(\Bbb R)$ be a complex-valued function and assume that 
\begin{equation}\label{eigenv}
(W_V(b)\psi)(x)=\psi(x+\ii b)+\psi(x-\ii b)-V(x)\psi(x) = \lambda \psi(x)\,.
\end{equation}
Let 
\begin{equation}\label{XY}
X = |V|^{1/2} \quad {\rm and} \quad  Y = V |V|^{-1/2}.
\end{equation}
Then the Birman--Schwinger principle states that the operator 
$Y G_\lambda X$ has an eigenvalue 1 and hence its operator norm is greater or equal to 1. Using \eqref{rezolvent} we find that the integral kernel of this operator equals
$$
Y(x) \frac{1}{2b\sin(\omega)}\, \frac{\sinh \left(\frac{\omega}{b} (x-y)\right)}{\sinh \left(\frac{\pi}{b} (x-y)\right)} X(y)
$$
and hence using Proposition \ref{G_ineq} we obtain
\begin{multline*}
\left| \left( \psi, Y G_\lambda X \varphi\right)\right| \le
\sup_{x\in\Bbb R} \left|G_\lambda(x)\right| \,\|V\|_1 \, \|\psi\|_2 \,\|\varphi\|_2\\
\le
 \left|G_\lambda(0)\right| \|V\|_1 \, \|\psi\|_2 \,\|\varphi\|_2
=  \left|\frac{1}{2\pi b} \, \frac{\omega}{\sin(\omega)}\right| \, \|V\|_1 \, \|\psi\|_2 \,\|\varphi\|_2.
\end{multline*}
Thus 
\begin{equation*}\label{1.1}
\left| \frac{\sin(\omega)}{\omega}\right| \le \frac{1}{2\pi b} \int_{\Bbb R} |V(x)|\, dx
\end{equation*}
and this proves \eqref{main}.

\medskip
\noindent
In order to prove that the constant in the inequality  \eqref{main} is sharp we consider the potential 
$V_c(x) = c \delta(x)$, where $\delta$ is the Dirac $\delta$-function and $c\in \Bbb C\setminus [0,\infty)$. The potential $V_c$ is a rank one perturbation of the ``free" operator $W_0(b)$. 
In Fourier space the eigenequation becomes
\begin{equation}\label{1rank_op}
2\cosh(2\pi k )\,\widehat\psi_c (k) - c\psi_c(0) = \lambda \widehat\psi_c(k).
\end{equation}
Denoting as before $\lambda = -2\cos(\omega)$, $\omega\in\Omega$, we obtain
$$
\widehat\psi_c(k) = \frac{c\psi_c(0)}{2\cosh(2\pi k) + 2 cos(\omega)}.
$$
Therefore 
\begin{equation}\label{eigenfunc1}
\psi_c(x) = c \psi_c(0) G_{-2\cos(\omega)} (x) = \frac{c\psi_c(0)}{2b\sin(\omega)}\,\,  
\frac{\sinh\left(\frac{\omega}{b} x \right)}{\sinh\left(\frac{\pi}{b} x \right)}.
\end{equation}
Letting $x\to 0$ in the last identity we find 
$$
1= \frac{c}{2b\sin(\omega)}\,\, \frac{\omega}{\pi} 
$$
and since $c=\int V_c\, dx$ we conclude that 
\begin{equation*}
\frac{\sin(\omega)}{\omega} = \frac{1}{2\pi b} \int_{\Bbb R} V_c(x) \, dx.
\end{equation*}
The proof of Theorem \ref{1} is complete.


\setcounter{equation}{0}
\section{Examples}\label{Sec:4}

\medskip
\noindent
Let us consider the equation 
$$
W_0(b) u(x) - c\delta(x)  u(x) = \lambda u(x),
$$
where $c= r e^{i\vartheta}$ with $r>0$ and $\vartheta\in[0,2\pi)$. For simplicity we assume that $b=1$. 
Then the eigenfunction \eqref{eigenfunc1} becomes
\begin{equation}\label{eigenfunc2}
\psi_c(x) = \frac{c\psi_c(0)}{2\sin(\omega)}\,\,  
\frac{\sinh\left(\omega \,x \right)}{\sinh\left(\pi \,x \right)},
\end{equation}
and $\psi_c$ is in $L^2(\Bbb R)$ for $\re\omega\in[0,\pi)$, where it is also an analytic function of $\omega$.
However, this function has singularities on the complex line $\omega= \pi + it$, $t\in \Bbb R$, and is exponentially growing if $\re\omega>\pi$. Therefore the equation 
\begin{equation}\label{1rank_eq}
\frac{\sin(\omega)}{\omega} =   \frac{r }{2\pi}\, e^{i\vartheta}
\end{equation}
defines the eigenvalues $\lambda = -2\cos(\omega)$ only under the assumption ${\rm Re}\, \omega\in [0,\pi)$.  However the equation \eqref{1rank_eq} can be solved even for ${\rm Re}\, \omega>\pi$ and thus gives infinitely many solutions \eqref{eigenfunc2} to the corresponding eigenequation that are not in $L^2(\Bbb R)$. It is natural to identify the latter values of $\lambda$ with resonances.  

\noindent
Below we present graphs for three different coupling constants  
$r/2\pi$,  namely $r/2\pi = 2, 0.25$ and $ 0.2$. We plot the solutions $\omega$ of \eqref{1rank_eq} for $\vartheta\in[0,2\pi)$ with 
$$
{\rm Re}\,\omega\in[0,\pi), \quad {\rm Re}\,\omega\in[\pi, 2\pi) \quad {\rm and} \quad {\rm Re}\,\omega\in[2\pi, 3\pi).
$$
In each of the plots we highlight the solutions obtained for $\vartheta=\frac{k\pi}{4}$ where $k=0,\dots,7$. We also plot the corresponding values $-2\cos(\omega)$. The complex eigenvalues are given by only the violet curves and the blue and green curves are resonances. In particular, in all three cases we note the absence of a complex eigenvalue if $\vartheta$ is sufficiently close to $\pi$. 

\vspace{5cm}

\begin{figure}[!htbp]
\centerline{\includegraphics[scale=.37]{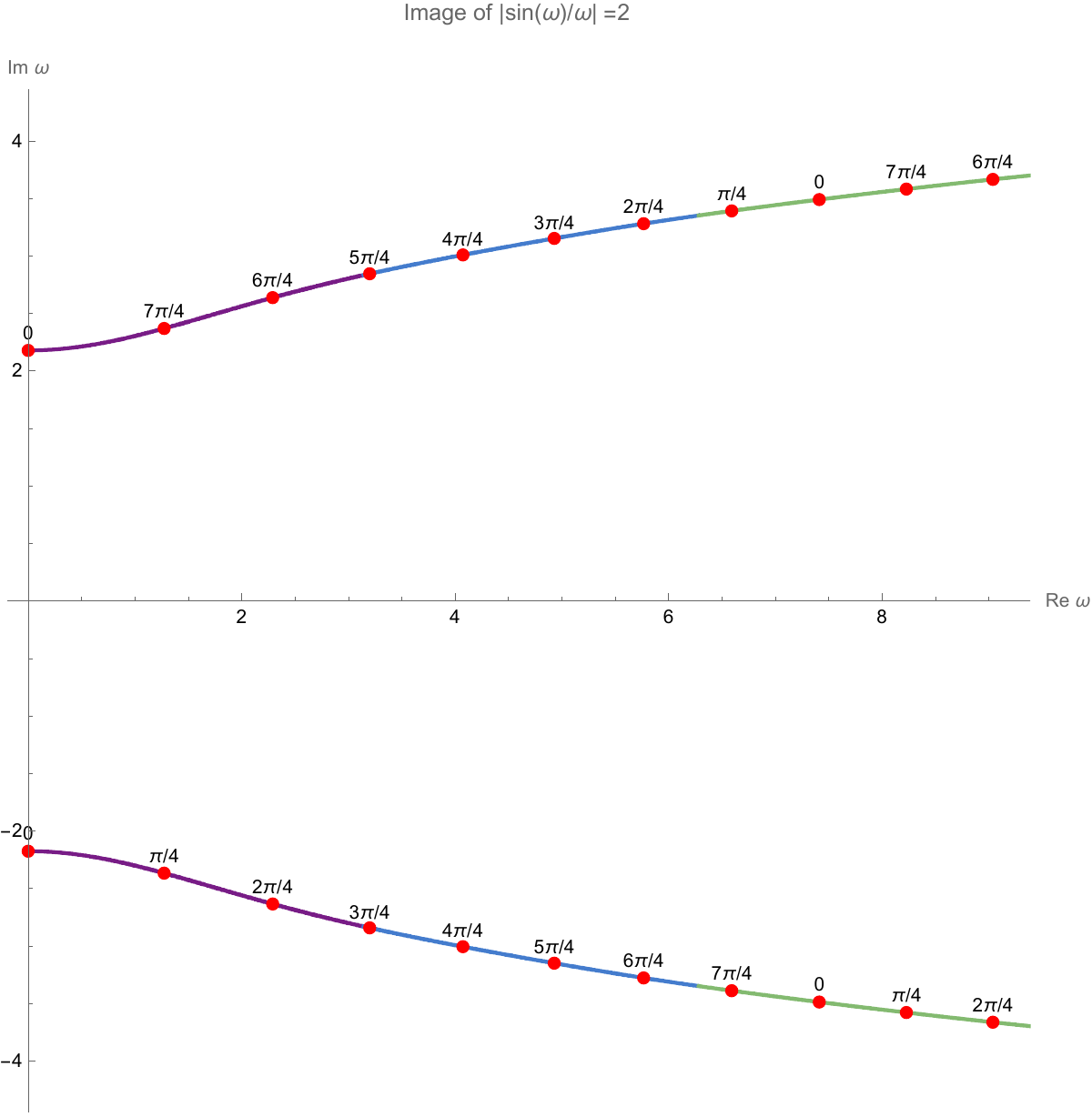}\quad   \includegraphics[scale=.41]{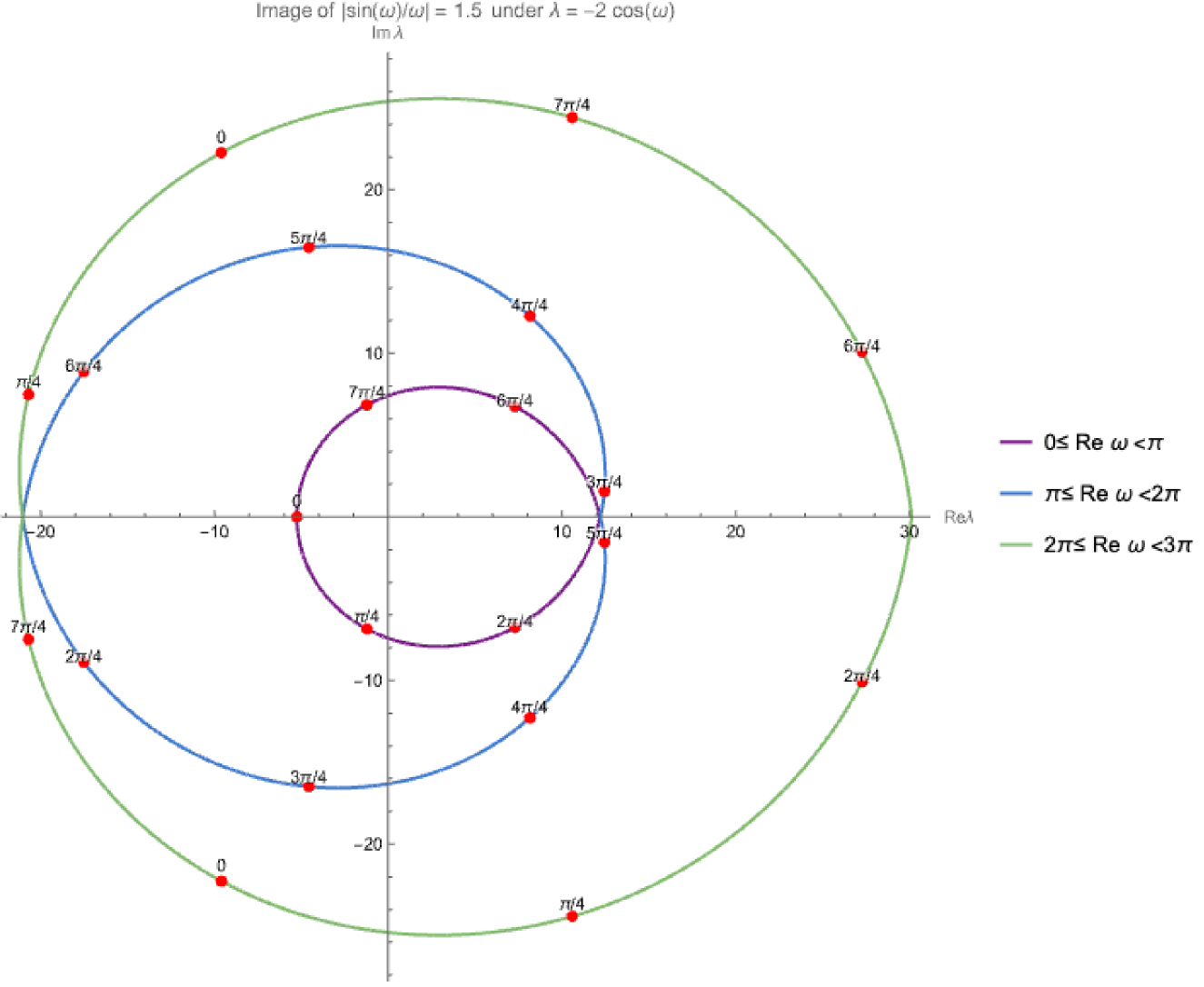}}
\caption{The solutions $\omega$ and $-2\cos\omega$ for $r/2\pi=2$.}
\end{figure}
\begin{figure}
\centerline{\includegraphics[scale=.37]{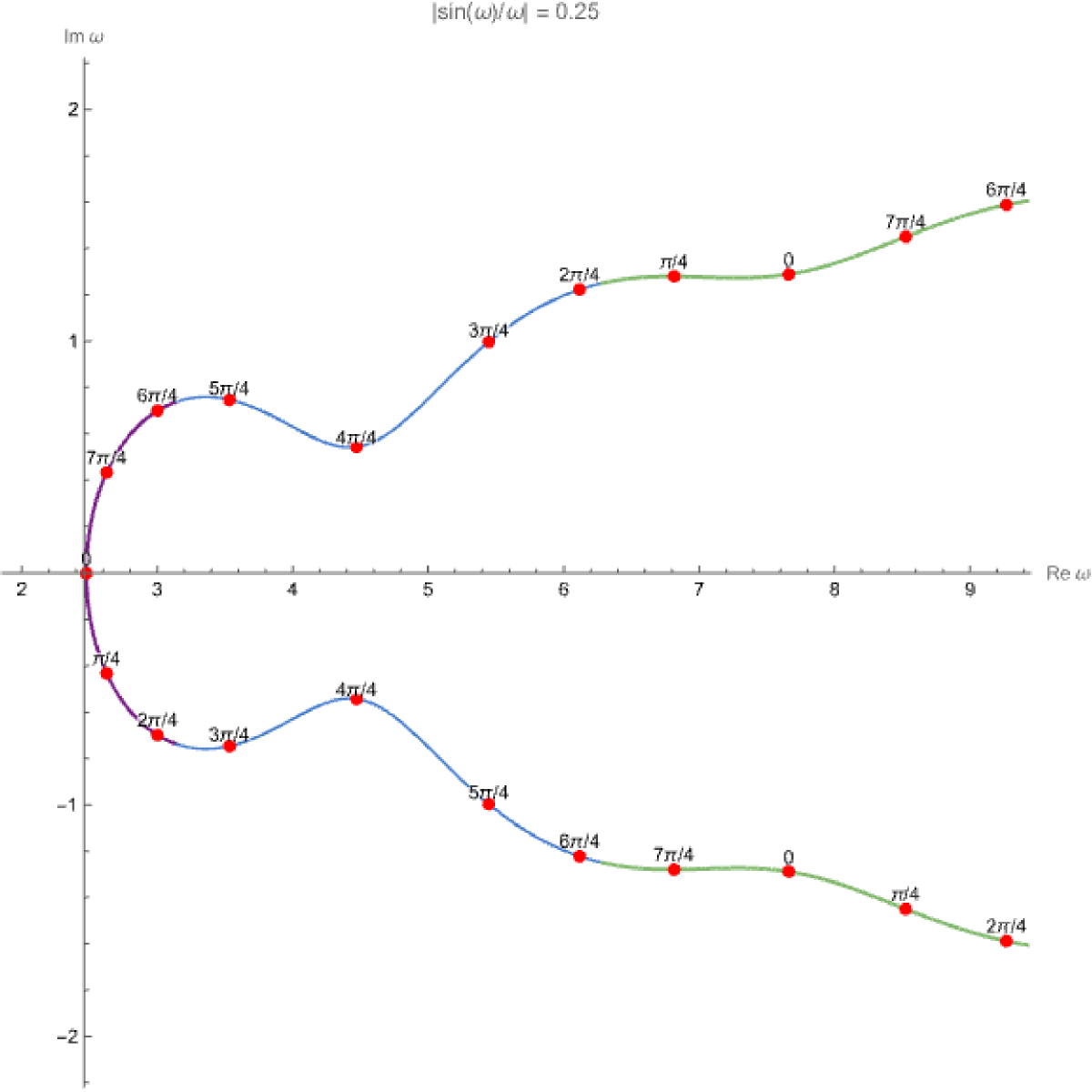}\quad   \includegraphics[scale=.41]{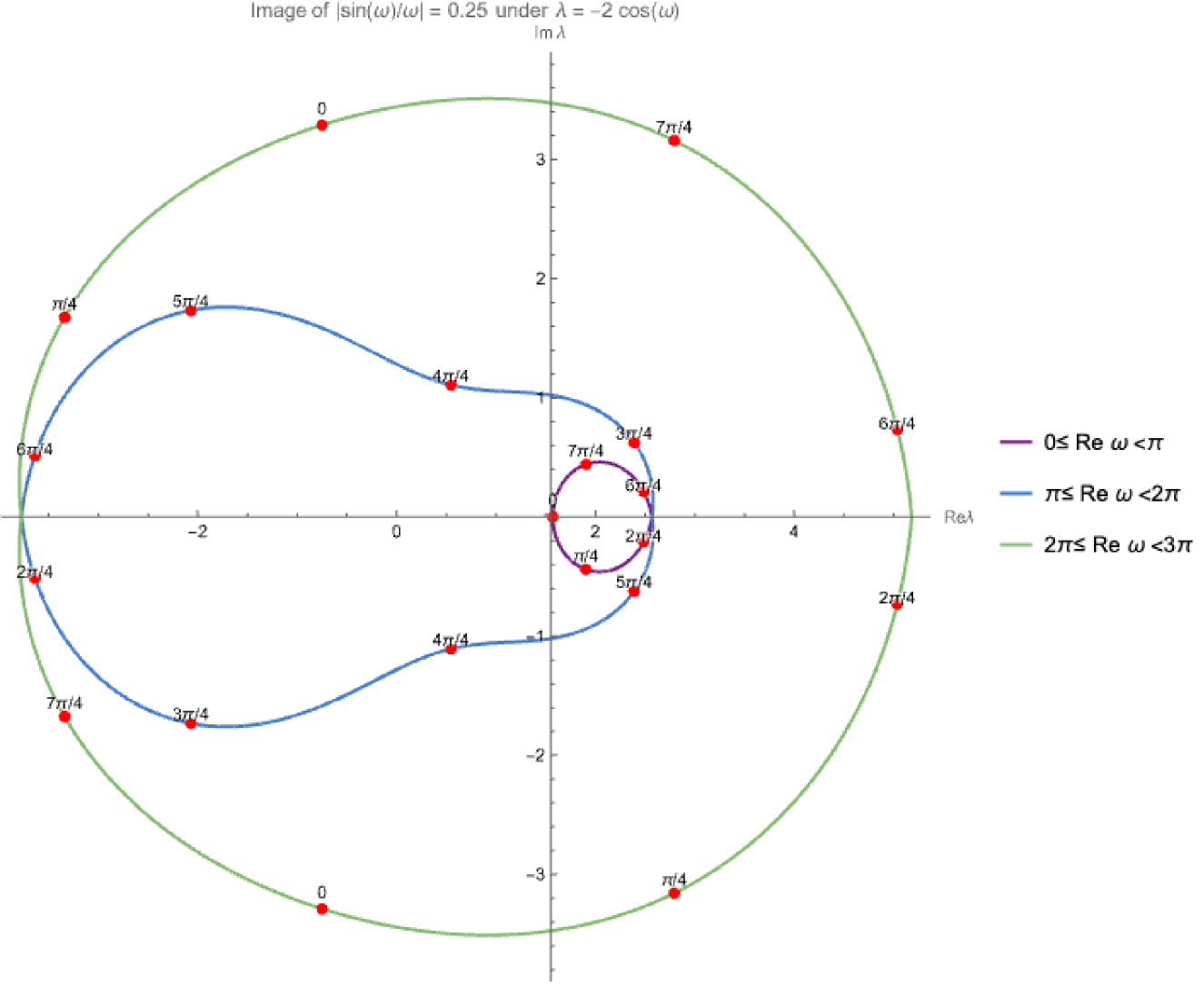}}
\caption{The solutions $\omega$ and $-2\cos\omega$ for $r/2\pi=0.25$.}
\end{figure}
\begin{figure}
\centerline{\includegraphics[scale=.37]{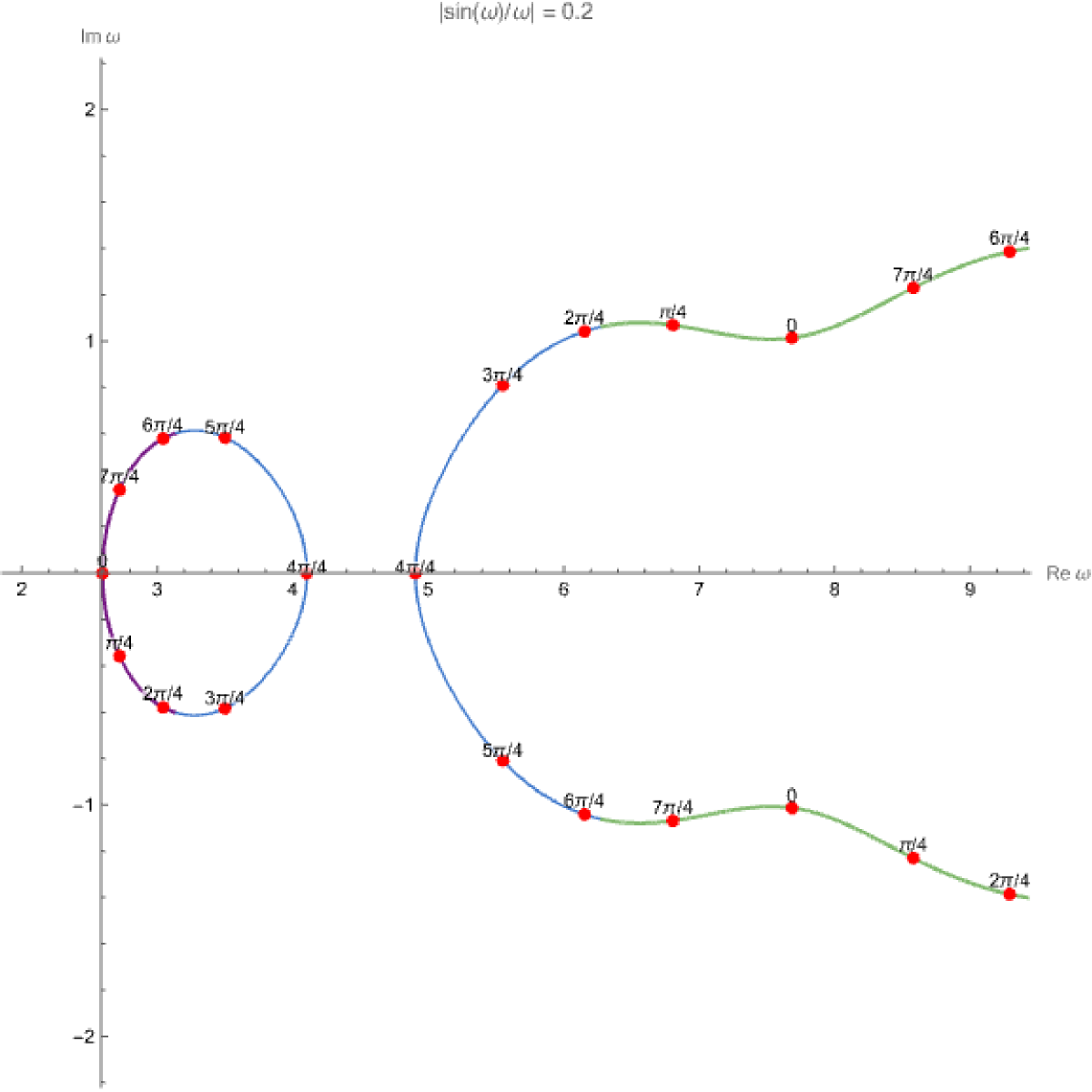}\quad   \includegraphics[scale=.41]{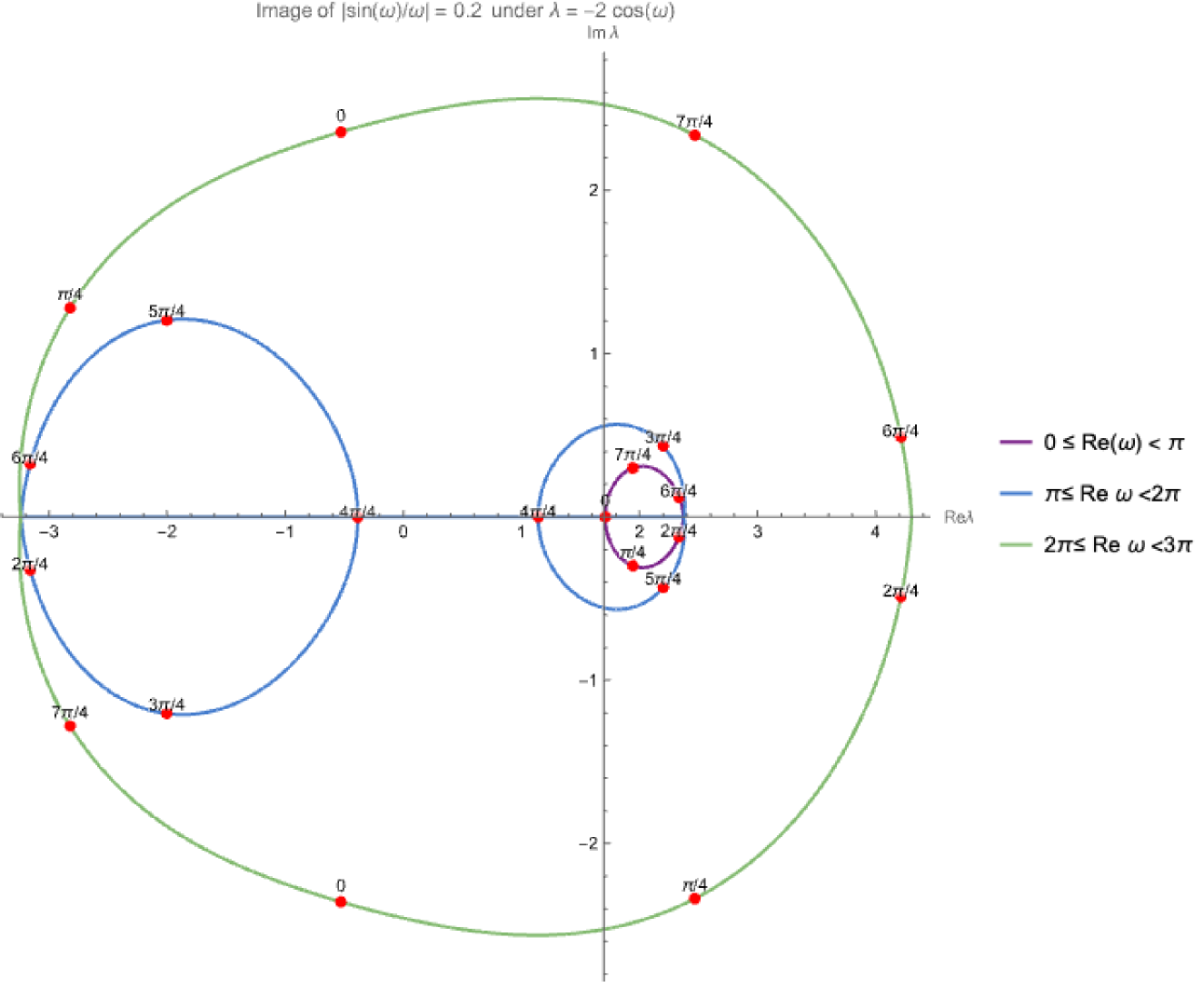}}
\caption{The solutions $\omega$ and $-2\cos\omega$ for $r/2\pi=0.2$.}
\end{figure}


\end{document}